\documentclass[12pt]{amsart}
\usepackage{amssymb,xcolor,graphics}
\usepackage[psdextra]{hyperref}
\usepackage{tabularx}
\usepackage{booktabs}
\usepackage{caption}
\usepackage{mathdots}

\newtheorem{thm}{Theorem}[section]
\newtheorem{lem}[thm]{Lemma}
\newtheorem{cor}[thm]{Corollary}
\newtheorem{prop}[thm]{Proposition}
\newtheorem{ex}[thm]{Example}

\newtheorem*{prob*}{Open problem}

\theoremstyle{definition}
\newtheorem{defi}[thm]{Definition}

\theoremstyle{remark}
\newtheorem*{rem*}{Remark}

\DeclareMathOperator{\id}{id}

\DeclareMathOperator{\rad}{rad}

\DeclareMathOperator{\Aut}{Aut}

\newcommand{\kringel}{\mathbin{\raise1pt\hbox{$\scriptstyle\circ$}}}
\newcommand{\pkt}{\mathbin{\raise0pt\hbox{$\scriptstyle\bullet$}}}

\newcommand{\C}{\mathbb{C}}

\newcommand{\ad}{{\rm ad}}

\newcommand{\End}{{\rm End}}
\newcommand{\Der}{{\rm Der}}

\newcommand{\nil}{\mathop{\rm nil}}

\newcommand{\La}{\mathfrak{a}}
\newcommand{\Lb}{\mathfrak{b}}

\newcommand{\Lf}{\mathfrak{f}}
\newcommand{\Lg}{\mathfrak{g}}
\newcommand{\Lh}{\mathfrak{h}}

\newcommand{\Ll}{\mathfrak{l}}
\newcommand{\Ln}{\mathfrak{n}}

\newcommand{\Ls}{\mathfrak{s}}

\newcommand{\CA}{\mathcal{A}}

\newcommand{\CN}{\mathcal{N}}
\newcommand{\CO}{\mathcal{O}}

\newcommand{\la}{\lambda}
\newcommand{\om}{\omega}

\newcommand{\ra}{\rightarrow}

\renewcommand{\phi}{\varphi}

\begin{document}

% Ab hier duerfen Sie wieder.

\title[Semisimple decompositions]{Semisimple decompositions of Lie algebras and prehomogeneous modules}
%  Die Kurzfassung kommt oben ueber die Seiten, sie steht in eckigen Klammern
%  Auch Autorennamen koennen eine Kurzfassung haben
\author[D. Burde]{Dietrich Burde}
\author[W. A. Moens]{Wolfgang Alexander Moens}
\address{Fakult\"at f\"ur Mathematik\\
Universit\"at Wien\\
Oskar-Morgenstern-Platz 1\\
1090 Wien \\
Austria}
\email{dietrich.burde@univie.ac.at}
\email{wolfgang.moens@univie.ac.at}
\date{\today}

%\thanks{The first author acknowledges support by the Austrian Science Foundation FWF, grant P28079 and grant I3248.}

\begin{abstract}
We study {\em disemisimple} Lie algebras, i.e., Lie algebras which can be written as a vector space sum of two
semisimple subalgebras. We show that a Lie algebra $\Lg$ is disemisimple if and only if its solvable radical
coincides with its nilradical and is a prehomogeneous $\Ls$-module for a Levi subalgebra $\Ls$ of $\Lg$. We use the
classification of prehomogeneous $\Ls$-modules for simple Lie algebras $\Ls$ given by Vinberg to show that the solvable
radical of a disemisimple Lie algebra with simple Levi subalgebra is abelian. We extend this result to disemisimple Lie algebras
having no simple quotients of type $A$.
 \end{abstract}

\maketitle

\section{Introduction}

Decompositions of groups and algebras have been studied by various authors in different contexts.
A triple $(G,A,B)$, where $G$ is a group and $A$ and $B$ are subgroups of $G$, is called a {\em decomposition} if $G=AB$. 
Similarly, a triple $(\Lg,\La,\Lb)$, where $\Lg$ is a Lie algebra and $\La$ and $\Lb$ are Lie subalgebras of $\Lg$, is called a
{\em decomposition} if $\Lg=\La+\Lb$. Here $\La+\Lb$ denotes the vector space sum of $\La$ and $\Lb$, which
need not be direct in general. The same definition also applies to arbitrary non-associative algebras. \\[0.2cm]
Wieland and Kegel \cite{KE1} proved that a finite group $G$ that is the product $G=AB$ of two nilpotent subgroups $A$ and $B$,
is solvable. A group $G$ is called {\em dinilpotent}, if it can be written as a product of two nilpotent subgroups $A$ and $B$.
Several further results on dinilpotent groups were obtained, e.g., on their derived length \cite{CST}.
For results on finite groups that can be written as the product of two simple groups, see \cite{WAL} and the references therein. \\[0.2cm]
Bahturin and Kegel \cite{BAK} completely described all finite-dimensional associative algebras over a field
that are sums of two simple subalgebras. Several results on sums of simple and nilpotent subalgebras of Lie algebras have
been obtained in \cite{BTT}. \\[0.2cm]
Onishchik \cite{ON62,ON69} studied decompositions of reductive Lie groups and Lie algebras in the context of 
transitive actions on homogeneous spaces. He classified all reductive decompositions for complex reductive 
Lie algebras. In particular he gave the decompositions of simple Lie algebras into the sum of two simple 
subalgebras. For example, one has
\[
\mathfrak{sl}_{2n}(\C)=\mathfrak{sl}_{2n-1}(\C)+\mathfrak{sp}_{2n}(\C)
\]
for all $n\ge 2$. Here the given vector space sum $\mathfrak{sl}_{2n}(\C)=\Ls_1+\Ls_2$ is not direct. In fact, we have
$\Ls_1\cap \Ls_2\cong \mathfrak{sp}_{2n-1}(\C)$ for the subalgebras $\Ls_1\cong \mathfrak{sl}_{2n-1}(\C)$ and
$\Ls_2\cong \mathfrak{sp}_{2n}(\C)$ in this decomposition. \\[0.2cm]
Decompositions of Lie algebras very often also arise in the context of post-Lie algebras and post-Lie algebra structures on
pairs of Lie algebras. For example, in \cite{BU41} we have shown that for any pair $(\Lg,\Ln)$ admitting a post-Lie algebra structure
with $\Lg$ nilpotent, $\Ln$ must be solvable. The proof relies on the decomposition result by Goto \cite{GOT}, stating that
the sum of two nilpotent Lie subalgebras over the complex numbers is solvable. For related results on post-Lie algebra structures
see also \cite{BU44,BU51}. Furthermore, post-Lie algebra structures are closely related to Rota-Baxter operators, which in turn
are naturally related to decomposition results of Lie algebras, see \cite{BU59,BU64}. \\
For rigidity results in the context of post-Lie algebra structures on pairs $(\Lg,\Ln)$ of semisimple respectively reductive Lie algebras,
decomposition results on the sum of two semisimple Lie algebras are of great importance. In particular, the results of Onishchik mentioned
above can be used. However, in these results only decompositions of reductive Lie algebras are studied.
We want to study more generally {\em disemisimple} Lie algebras, i.e., arbitrary Lie algebras which can be written as a vector
space sum of two semisimple subalgebras. \\[0.2cm]
The paper is organized as follows. In section $2$ we introduce disemisimple Lie algebras and state some basic results, such as that
every semisimple Lie algebra is disemisimple, but the converse need not be true. For example,
\[
\Ls\Ll_n(\C)\ltimes V(n)=\Ls\Ll_n(\C)+\Ls\Ll_n(\C)
\]
is disemisimple, where $V(n)$ denotes the natural representation of $\Ls\Ll_n(\C)$, but not semisimple. We show that a disemisimple
Lie algebra is perfect and thus its solvable radical is nilpotent. We establish a direct connection between disemisimple Lie algebras
and prehomogeneous modules in Theorem $\ref{2.7}$. We state Vinberg's classification \cite{VIN} of prehomogeneous $\Ls$-modules for complex
simple Lie algebras $\Ls$, and also show by an example, that we can answer the question, whether a given $\Ls$-module is prehomogeneous
or not by a constructive algorithm, not only for simple $\Ls$, but also for semisimple $\Ls$. Here we also have the classification of
Sato and Kimura \cite{SAK} for irreducible $\Ls$-modules of semisimple Lie algebras $\Ls$. However, this classification
is not as explicit as Vinberg's classification for simple $\Ls$, because it involves castling classes and strong equivalence of
prehomogeneous modules. So it is difficult in general to decide whether or not a module for a given example is prehomogeneous. \\[0.2cm]
In section $3$ we show how to use Vinberg's classification
to prove that the solvable radical of a disemisimple Lie algebra with a simple Levi subalgebra is {\em abelian}, see Theorem $\ref{3.7}$.
For this, we prove a result relating the existence of a disemisimple Lie algebra with a non-abelian solvable radical to the existence of
prehomogeneous modules of a certain type involving the tensor product and the wedge product of irreducible $\Ls$-modules.
The question remains whether or not the solvable radical of any disemisimple Lie algebra is
abelian. Our proof for the case that we have a simple Levi subalgebra cannot be extended this way to the case where we have only a
semisimple Levi subalgebra. Still, we can use the classification by Sato and Kimuara to show that the solvable radical of
any disemisimple $A$-free Lie algebra is abelian, see Theorem $\ref{4.3}$. Here $A$-free means that the Lie algebra has no
simple quotients of type $A$, i.e., not isomorphic to a Lie algebra $\mathfrak{sl}_n(\C)$ for any $n\ge 2$.

\section{Disemisimple Lie algebras}

All Lie algebras are finite-dimensional and assumed to be complex. In analogy to dinilpotent groups
and Lie algebras we introduce the following name for the class of Lie algebras, which admit a decomposition into two semisimple
subalgebras.

\begin{defi}\label{2.1}
A Lie algebra $\Lg$ is called {\em disemisimple}, if it can be written as a vector space sum of two semisimple subalgebras
$\Ls_1$ and $\Ls_2$ of $\Lg$. In this case we write $\Lg=\Ls_1+\Ls_2$. 
\end{defi}  

Note that the vector space sum need not be direct, and that the subalgebras need not be ideals.
For the direct vector space sum we will write  $\Lg=\Ls_1\dotplus \Ls_2$, in order to distinguish it from the
direct sum of Lie algebras $\Lg=\Ls_1\oplus \Ls_2$. \\[0.2cm]
Of course, every semisimple Lie algebra $\Ls$ is disemisimple. One possible decomposition is $\Ls=\Ls+\Ls$.
The converse however need not be true in general.  Not every disemisimple Lie algebra is semisimple.
Indeed, we have the following counterexample, see Example $4.10$ of \cite{BU64}.

\begin{ex}\label{2.2}
Let $V(n)$ be the natural $\mathfrak{sl}_n(\C)$-module of dimension $n\ge 2$ and $\Lg=\Ls\Ll_n(\C)\ltimes V(n)$ be the semidirect
product, where $V(n)$ is considered as abelian Lie algebra. Let $\phi$ be the automorphism given by $\phi=e^{\ad(z)}$
for a certain $z$ in the solvable radical $V(n)$. Then  
\[
\Ls\Ll_n(\C)\ltimes V(n)=\Ls\Ll_n(\C)+\phi(\Ls\Ll_n(\C))
\]
is the vector space sum of two semisimple subalgebras. The Lie algebra $\Lg$ is perfect, but not semisimple.
\end{ex}

In general we write $\Lg=\Ls\ltimes \rad(\Lg)$ for a Levi decomposition, where $\rad(\Lg)$ is the solvable radical of $\Lg$.
When we write it  as an $\Ls$-module like $V(n)$ in the above example, then we consider it as an abelian Lie algebra.
We denote by $\nil(\Lg)$ the nilradical of $\Lg$. \\[0.2cm]
In the following we want to study the structure of disemisimple Lie algebras in more detail.

\begin{lem}\label{2.3}
Let $\Lg=\Ls_1 + \Ls_2$ be a disemisimple Lie algebra. Then $\Lg$ is perfect and the solvable radical of $\Lg$
coincides with the nilradical.
\end{lem}  

\begin{proof}
We have
\begin{align*}
[\Lg,\Lg] & = [\Ls_1+\Ls_2, \Ls_1+\Ls_2] \\
          & = [\Ls_1,\Ls_1]+[\Ls_2,\Ls_2]+ [\Ls_1,\Ls_2] \\
          & = \Ls_1+\Ls_2+ [\Ls_1,\Ls_2] \\
          & = \Lg.
\end{align*}
Denote by $\Ls$ a Levi subalgebra of $\Lg$. Then we have
\[
\Ls\dotplus \rad(\Lg)=\Lg=[\Lg,\Lg]=[\Ls\dotplus \rad(\Lg), \Ls\dotplus \rad(\Lg)]=\Ls \dotplus [\Ls,\rad(\Lg)].
\]
It follows that $\rad(\Lg)=[\Ls,\rad(\Lg)]\subseteq [\Lg,\rad(\Lg)]\subseteq \nil(\Ln)$. In the last step we have used
that $D(\rad(\Lg))\subseteq \nil(\Ln)$ for all derivations $D\in \Der(\Lg)$, and hence in particular for $D=\ad(x)$.
It follows that $\rad(\Lg)$ is nilpotent.
\end{proof}  

It turns out that decompositions of Lie algebras and disemisimple Lie algebras are related to prehomogeneous modules. 
Let $G$ be a connected complex algebraic group. Then a rational $G$-module $V$ is called {\em prehomogeneous} if $V$ has
a Zariski-open $G$-orbit $\CO$. It is called {\em \'etale} if in addition $\dim (\CO)=\dim (G)$, see \cite{SAK,BU53}.
If $G$ is semisimple with Lie algebra $\Lg$, then a $G$-module $V$ given by a representation $\rho$ is prehomogeneous
if and only if $d\rho$ is a prehomogeneous $\Lg$-module. This is defined for Lie algebras as follows.  

\begin{defi}\label{2.4}
Let $\Lg$ be a Lie algebra. A $\Lg$-module $V$ is called {\em prehomogeneous} if there exists a vector $v\in V$
such that $\Lg\cdot v=V$. It is called {\em \'etale} if in addition $\dim (V)=\dim (\Lg)$.
\end{defi}

The evaluation map $e_v\colon \Lg\ra V$ given by $x\mapsto x\cdot v$ is linear. So we always have
\[
\dim (V)\le \dim (\Lg)
\]
for a prehomogeneous $\Lg$-module. \\[0.2cm]
Note that the zero module is prehomogeneous. Vinberg classified prohomogeneous $\Ls$-modules of simple
complex Lie algebras $\Ls$ in \cite{VIN}. He used the term {\em transitive} for prehomogeneous. Sato and Kimura gave
a classification of rational prehomogeneous $G$-modules for connected semisimple algebraic groups. For the irreducible case see
\cite{SAK}, Proposition $1$ on page $143$ and $144$. This classification is up to castling equivalence. \\[0.2cm]
For semisimple Lie algebras the following result is known for \'etale $\Ls$-modules, see for example
\cite{BU53}, Corollary $4.5$.

\begin{prop}\label{2.5}
Let $\Ls$ be a semisimple Lie algebra over a field of characteristic zero. Then $\Ls$ admits no \'etale $\Ls$-module.
\end{prop}  

This result has been stated earlier in the literature, but in different contexts and with different notations.
For the result in terms of left-symmetric algebras, see \cite{HEL}, Corollary $21$.

\begin{cor}\label{2.6}
Let $\Ls$ be a semisimple Lie algebra over a field of characteristic zero and $V$ be a nonzero prehomogeneous
$\Ls$-module. Then we have $2\le \dim (V)\le \dim (\Ls)-1$.   
\end{cor}  

\begin{proof}
The trivial $1$-dimensional module is not prehomogeneous. So we have $\dim (V)\ge 2$. We always have $\dim (V)\le \dim (\Ls)$
and equality is not possible by Proposition $\ref{2.5}$.
\end{proof}  

Our result concerning disemisimple Lie algebras and prehomogeneous modules is as follows.

\begin{thm}\label{2.7}
A Lie algebra $\Lg$ is disemisimple if and only if its solvable radical coincides with its nilradical and
is a prehomogeneous $\Ls$-module for a Levi subalgebra $\Ls$ of $\Lg$.
\end{thm}  

\begin{proof}
Assume that $\Lg$ is disemisimple. Then there exist semisimple subalgebras $\Ls_1,\Ls_2$ of
$\Lg$ with $\Lg=\Ls_1+\Ls_2$. Since each semisimple subalgebra of $\Lg$ is contained in a maximal
semisimple subalgebra, we may assume that $\Ls_1$ and $\Ls_2$ are {\em Levi subalgebras} of $\Lg$.
According to Lemma $\ref{2.3}$, $\Lg$ is perfect and $\rad(\Lg)=\nil(\Lg)$. So it remains to show that $\nil(\Lg)$
is a prehomogeneous module for, say, the Levi subalgebra $\Ls_1$. By the theorem of Levi-Malcev there exists an element
$z\in \nil(\Lg)$, such that the automorphism $\phi=e^{\ad (z)}\in \Aut(\Lg)$ maps $\Ls_1$ onto $\Ls_2$. So $\phi(s)$ for
$s\in \Ls_1$ is given explicitly by
\[
\phi(s)=s+[z,s]+\frac{1}{2!}[z,[z,s]]+ \frac{1}{3!}[z,[z,[z,s]]]+ \cdots,
\]
and we have
\[
\Lg=\Ls_1+\Ls_2=\Ls_1+\phi(\Ls_1)=\Ls_1+\Ls_1+ (\phi-\id)(\Ls_1)=\Ls_1+(\phi-\id)(\Ls_1).
\]  
Since $\nil(\Lg)$ is an ideal, the formula for $\phi(s)$ implies that $(\phi -\id)(\Ls_1)\subseteq \nil(\Lg)$.
So we have the inclusions
\[
  \Ls_1 \dotplus \nil(\Lg)\subseteq \Lg\subseteq \Ls_1\dotplus (\phi-\id)(\Ls_1)\subseteq \Ls_1\dotplus \nil(\Lg),
\]  
which imply that
\[
\nil(\Lg)=(\phi-\id)(\Ls_1).
\]  
Define a linear map $\psi\colon \nil(\Lg)\rightarrow \nil(\Lg)$ by
\[
\psi=\id+\frac{1}{2!}\ad(z)+\frac{1}{3!}\ad(z)^2+\cdots .
\]
Since $\ad(z)$ is a nilpotent endomorphism of $\nil(\Lg)$ we see that $\psi$ is a vector space automorphism
of $\nil(\Lg)$. We have $(\phi -\id)(s)=\psi([s,-z])$ for all $s\in \Ls_1$. This implies
\[
\nil(\Lg)=(\phi-\id)(\Ls_1)=\psi([\Ls_1,-z]),
\]
and hence $\nil(\Lg)=\psi^{-1}(\nil(\Lg))=[\Ls_1,-z]$. In other words, $\nil (\Lg)$ is a prehomogeneous $\Ls_1$-module
for the Levi subalgebra $\Ls_1$ under the natural action $s\cdot v=[s,v]$ for $s\in \Ls_1$ and $v\in \nil(\Lg)$. \\[0.2cm]
Conversely assume that $\rad(\Lg)=\nil(\Lg)$, and that this is a prehomogeneous $\Ls_1$-module for a Levi subalgebra
$\Ls_1$ of $\Lg$. Let $z$ be an element in $\nil(\Lg)$ such that $\nil(\Lg)=[\Ls_1,-z]=[z,\Ls_1]$ and define an automorphism
of $\Lg$ by $\phi=e^{\ad(z)}$. Let $\Ls_2=\phi(\Ls_1)$. Then $\Ls_2$ also is a Levi subalgebra of $\Lg$ and we have,
with $\psi$ as above,
\begin{align*}
\Ls_1+\Ls_2 & = \Ls_1+(\phi-\id)(\Ls_1) \\
            & = \Ls_1+\psi ([z,\Ls_1]) \\
            & = \Ls_1+\psi(\nil(\Lg)) \\
            & = \Ls_1+\nil(\Lg) = \Lg.
\end{align*}  
\end{proof}

We obtain the following corollary to Theorem $\ref{2.7}$.

\begin{cor}
The disemisimple Lie algebras $\Lg$ whose solvable radical is abelian are precisely the Lie algebras which are isomorphic
to a semidirect sum  $\Lg=\Ls\ltimes V$, where $\Ls$ is a Levi subalgebra of $\Lg$ and $V$ is a prehomogeneous $\Ls$-module.
Then $\Lg=\Ls+\phi(\Ls)$ with $\phi=e^{\ad(z)}$ for some $z\in V$.
\end{cor}  

Recall that $V=0$ is a prehomogeneous $\Ls$-module. This yields again the trivial fact that every semisimple Lie algebra
$\Ls$ is disemisimple with $\Ls=\Ls+\Ls$. We also recover Example $\ref{2.2}$, because for $\Lg=\mathfrak{sl}_n(\C)\ltimes V(n)$
the $\mathfrak{sl}_n(\C)$-module $V(n)$ is prehomogeneous. \\[0.2cm]
Now let us consider the case where $\Lg$ has a {\em simple} Levi subalgebra. As we already mentioned, Vinberg \cite{VIN}
classified all prehomogenous $\Ls$-modules for complex simple Lie algebras $\Ls$. For a given type of $\Ls$, denote by
$L(\om_i)$ the irreducible highest weight module with fundamental weight $\om_i$. Note that $V$ is prehomogeneous if and only if
its dual module $V^*$ is prehomogeneous.

\begin{thm}[Vinberg]\label{2.9}
Let $\Ls$ be a complex simple Lie algebra and $V\neq 0$ be an irreducible $\Ls$-module. Then $V$ is prehomogeneous for
$\Ls$ if and only if $V$ is contained in the following list.
\vspace*{0.5cm}
\begin{center}
\begin{tabular}{c|ccc}
$\Ls$ & $\dim(\Ls)$ & $V$ & $\dim(V)$  \\[4pt]
\hline
$A_{\ell}$, $\ell\ge 1$ & $\ell(\ell+2)$  & $L(\om_1),L(\om_{\ell})$ & $\ell+1$  \\[4pt]
$A_{2\ell}$, $\ell\ge 2$ & $4\ell(\ell+1)$  & $L(\om_2),L(\om_{2\ell-1})$ & $\ell(2\ell+1)$  \\[4pt]
$C_{\ell}$, $\ell\ge 2$ & $\ell(2\ell+1)$  & $L(\om_1)$ & $2\ell$  \\[4pt]
$D_5$ & $45$  & $L(\om_4),L(\om_5)$ & $16$  \\[4pt]
\end{tabular}
\end{center}
\vspace*{0.5cm}
\end{thm}  
This is Theorem $7$ in \cite{VIN}. Vinberg' notation relates to our notation by
\[
  T_1=L(\om_1), \; T_1'=L(\om_{\ell}),\; T_2=L(\om_2), \; T_2'=L(\om_{\ell-1}),
\]
where he denotes the dual module of $V$ by $V'$. So $L(\om_1)^*\cong L(\om_{\ell})$ and $L(\om_2)^*\cong L(\om_{\ell-1})$. \\[0.2cm]
Vinberg also classified the reducible prehomogeneous $\Ls$-modules in Theorem $8$
of \cite{VIN}.

\begin{thm}[Vinberg]\label{2.10}
Let $\Ls$ be a complex simple Lie algebra and $V\neq 0$ be a reducible $\Ls$-module. Then $V$ is prehomogeneous for
$\Ls$ if and only if $V$ is contained in the following list.
\vspace*{0.5cm}
\begin{center}
\begin{tabular}{c|ccc}
$\Ls$ & $\dim(\Ls)$ & $V$ & $\dim(V)$  \\[4pt]
\hline
$A_{\ell}$, $\ell\ge 2$ & $\ell(\ell+2)$  & $mL(\om_1), \, mL(\om_{\ell}),\, 2\le m\le \ell$ & $m(\ell+1)$  \\[4pt]
$A_{2\ell}$, $\ell\ge 2$ & $4\ell(\ell+1)$  & $L(\om_1)\oplus L(\om_{2\ell-1}), \, L(\om_2)\oplus L(\om_{2\ell})$
                        & $(\ell +1)(2\ell+1)$  \\[4pt]
 & $4\ell(\ell+1)$  & $L(\om_2)\oplus L(\om_2), \, L(\om_{2\ell-1})\oplus L(\om_{2\ell-1})$
                        & $2\ell(2\ell+1)$  \\[4pt]
\end{tabular}
\end{center}
\vspace*{0.5cm}
\end{thm}

Note that in particular only $V=0$ is a prehomogeneous $\Ls$-module,  where $\Ls_1$ is
a simple Levi subalgebra of $\Lg$ of type $B_{\ell}$ for $\ell\ge 3$, or of type $D_{\ell}$ for $d=4$ and $d\ge 6$,
or of exceptional type. Then we just have the decomposition $\Lg=\Ls_1=\Ls_1+\Ls_1$. \\[0.2cm]
So far we only have considered the $\Ls$-module $V$ to be an abelian Lie subalgebra when writing $\Lg=\Ls\ltimes V$.
However, we can also use Vinberg's classification when the radical is possibly not abelian.
With $\Lg=\Ls\ltimes \nil(\Lg)$ being a Levi decomposition of a disemisimple Lie algebra, $\nil(\Lg)$ a priori need not be abelian
as a Lie algebra. If $\Ls$ is simple, we can again view the vector space $V$ of $\nil(\Lg)$ as an $\Ls$-module and
decide whether or not it is prehomogeneous by Vinberg's classification. We state this explicitly as follows.

\begin{prop}\label{2.11}
A Lie algebra $\Lg$ with simple Levi subalgebra $\Ls$ is disemisimple if and only if it is isomorphic to $\Ls\ltimes \nil(\Lg)$,
where $\nil(\Lg)$ as an $\Ls$-module is one of the prehomogeneous modules of Vinberg's classification.  
\end{prop}  

This description is also helpful for deciding whether or not a given low-dimensional Lie algebra is disemisimple.
We have the following result.

\begin{cor}
The only complex disemisimple Lie algebras of dimension $n\le 6$ are  $\mathfrak{sl}_2(\C)$,  $\mathfrak{sl}_2(\C)\ltimes V(2)$
and $\mathfrak{sl}_2(\C) \oplus \mathfrak{sl}_2(\C)$.
\end{cor}

\begin{proof}
Assume that $\Lg$ is disemisimple of dimension $n\le 6$ and not semisimple.  Then $\Lg$ is perfect by Lemma $\ref{2.3}$
and $\rad(\Lg)$ is nilpotent. Hence by the classification of perfect Lie algebras in low dimension, $\Lg$ is isomorphic to
one of the following Lie algebras:
\[
\mathfrak{sl}_2(\C) \ltimes V(2),\, \mathfrak{sl}_2(\C)\ltimes V(3),\, \mathfrak{sl}_2(\C)\ltimes \Ln_3(\C),
\]  
where $\Ln_3(\C)$ denotes the $3$-dimensional Heisenberg Lie algebra. We can apply Proposition $\ref{2.11}$, because
$\mathfrak{sl}_2(\C)$ is simple and there is no prehomogeneous  $\mathfrak{sl}_2(\C)$-module $V$ of dimension $3$ by
Vinberg's classification. On the other hand, the natural module $V(2)$ for $\mathfrak{sl}_2(\C)$ is prehomogeneous.
This yields the decomposition $\mathfrak{sl}_2(\C)\ltimes V(2)=\mathfrak{sl}_2(\C)+ \mathfrak{sl}_2(\C)$ from Example $\ref{2.2}$.
So we are done.
\end{proof}

The situation becomes much more complicated for the case that the Levi subalgebra $\Ls$ of $\Lg$ is only semisimple and
not simple. The classification of prehomogeneous $\Ls$-modules by Sato and Kimura here is not of such an
explicit nature. They consider triples $(\Ls,\rho, V(n))$ of semisimple Lie algebras $\Ls$ with a prehomogeneous representation
$\rho$, where the parameter $n$ in $V(n)$ denotes the dimension of the module. The classification of such triples then
only is up to {\em strong equivalence}, see Definition $4$ on page $36$ in \cite{SAK},  and {\em castling equivalence}, see
$10$ on page $39$ in \cite{SAK}. 
However in some cases we still can apply the classification list explicitly. We state Proposition $1$ of section $6$ in \cite{SAK}
in the language of Lie algebras as follows.

\begin{prop}[Sato, Kimura]\label{2.13}
Let $(\Ls,\rho, \dim(\rho))$ be a prehomogeneous triple, where $\Ls$ is a semisimple Lie algebra and $\rho$ is irreducible.
Then it is strongly equivalent to one of the following castling-reduced prehomogeneous triples.
\vspace*{0.5cm}
\begin{center}
\begin{tabular}{c|cccc}
$\Ls$ & $\dim(\Ls)$ & $\rho$ & $\dim(\rho)$ & conditions \\[4pt]
\hline
$\Ls_1\oplus A_m$ & $s+m(m+2)$  & $L(\la)\otimes L(\om_1)$ & $(n+1)(m+1)$ & $m>n>2$  \\[4pt]
$A_n \oplus A_m$  & $n(n+2)+m(m+2)$  & $L(\om_1)\otimes L(\om_1)$ & $(n+1)(m+1)$ & $m\ge 2n+1\ge 1$ \\[4pt]
$A_{2m}$  & $4m(m+1)$  & $L(\om_2)$ & $m(2m+1)$ & $m\ge 1$ \\[4pt]  
$A_1\oplus A_{2m}$ & $4m(m+1)+3$  & $L(\om_1)\otimes L(\om_2)$  & $2\cdot m(2m+1)$ & $m\ge 1$   \\[4pt]
$C_n \oplus A_m$  & $n(2n+1)+4m(m+1)$  & $L(\om_1)\otimes L(\om_1)$ & $(2n)(2m+1)$ & $n\ge m+1\ge 1$ \\[4pt]
$D_5$  & $45$  & $L(\om_4)$ & $16$ & $-$ 
\end{tabular}
\end{center}
\vspace*{0.5cm}
Here $\Ls_1$ is a semisimple Lie algebra of dimension $s$, different from $A_n$, and $L(\la)$ is any simple
$\Ls_1$-module of dimension $n+1$ for $n\ge 2$.
\end{prop}  

We can apply this result for some examples in low dimension.

\begin{ex}
Let $\Ls=\mathfrak{sl}_2(\C)\oplus \mathfrak{sl}_2(\C)$ and $V(2)=L(\om_1)$ be the natural $\mathfrak{sl}_2(\C)$-module.
Then the tensor module $V(2)\otimes V(2)$ for $\Ls$ is not prehomogeneous. However, $\Ls$ actually has
prehomogeneous $\Ls$-modules, for example $V=( V(1)\otimes V(2)) \oplus \left( V(2)\otimes V(1)\right)$.
\end{ex}

This can be verified directly. The first part also follows from the above classification. The module is $4$-dimensional
and irreducible. This does not change under strong equivalence. Also there is no castling equivalent triple with smaller
dimension for $\Ls$. Since the corresponding triple is not listed, the module is not prehomogeneous.  

\begin{ex}\label{2.15}
Let $\Ls=\mathfrak{sl}_2(\C)\oplus \mathfrak{sl}_3(\C)$ and $V=L(\om_1)\otimes L(\om_2)$ be the irreducible tensor module
of dimension $2\cdot 3=6$. Then $V$ is prehomogeneous and the Lie algebra $\Lg=\Ls\ltimes V$ of dimension $11+6=17$ is disemisimple
with $\Lg=\Ls+\phi(\Ls)$ for some automorphism $\phi$ of $\Ls$.
\end{ex}  

Indeed, the triple $(A_1\oplus A_2, L(\om_2)\otimes L(\om_1),6)$ is in the above list of Sato and Kimura for $m=1$.
Since $L(\om_2)^*\cong L(\om_1)$, and the
triples $(\Ls,\rho,\dim (\rho))$, $(\Ls,\rho^*,\dim (\rho))$ are strongly equivalent, also the $\Ls$-module
$V=L(\om_1)\otimes L(\om_1)$ is prehomogeneous. \\[0.2cm]
Not all examples can be found so easily in the list of Sato and Kimura because of castling and strong equivalence.
Moreover this list is restricted to irreducible modules. Hence it is interesting to show that, for an explicitly given
semisimple Lie algebra $\Ls$ and an $\Ls$-module $V$, there is a constructive algorithm to determine whether or not
$V$ is prehomogeneous. Here we will not give the algorithm in general, but rather illustrate it with Example $\ref{2.15}$.
This requires already some work. \\[0.2cm]
So let $V_2=L(\om_1)$ and $V_3=L(\om_2)$, with $V_2$ the natural $2$-dimensional $\mathfrak{sl}_2(\C)$-module,
and $V_3$ the natural $3$-dimensional $\mathfrak{sl}_3(\C)$-module. We want to show that the irreducible
$6$-dimensional $\Ls$-module $V=V_2\otimes V_3$ is prehomogeneous. In fact, we even explicitly construct a vector $v\in V$
such that $\Ls\cdot v=V$. First we parametrize the representation $\phi\colon \mathfrak{sl}_2(\C) \ra \End(V_2)$, $x\mapsto \phi(x)$
with respect to a basis $(e_1,e_2)$ of $V_2$ by
\[
\phi(x)=\begin{pmatrix} x_{11} & x_{12} \cr  x_{21} & x_{22} \end{pmatrix},
\]
where the entries are complex numbers with $x_{11}+x_{22}=0$. Similarly, we parametrize the representation
$\psi\colon \mathfrak{sl}_3(\C) \ra \End(V_3)$, $y\mapsto \psi(y)$
with respect to a basis $(f_1,f_2,f_3)$ of $V_3$ by
\[
\psi(y)=\begin{pmatrix} y_{11} & y_{12} & y_{13} \cr  y_{21} & y_{22} & y_{23} \cr y_{31} & y_{32} & y_{33} \end{pmatrix},
\]
where the entries are complex numbers with $y_{11}+y_{22}+y_{33}=0$. Now the action of $x+y$ of
$\Ls=\mathfrak{sl}_2(\C)\oplus \mathfrak{sl}_3(\C)$ on the vector space $V_2\otimes V_3$ with respect to the basis
\[
(e_1\otimes f_1,\, e_1\otimes f_2,\, e_1\otimes f_3,\, e_2\otimes f_1,\, e_2\otimes f_2,\, e_2\otimes f_3)
\]
is given by the Kronecker formula $\rho\colon \Ls\ra \End (V_2\otimes V_3)$, 
$(x+y)\mapsto \phi(x)\otimes \id_2+\id_3\otimes \psi(y)$,
where $\id_2,\id_3$ denote the identity matrices of size $2$ and $3$ respectively. Then we obtain the following explicit
parametrization
\[
\rho(x+y)= \begin{pmatrix} x_{11} & 0 & 0 & x_{12} & 0 & 0 \cr 0 & x_{11} & 0 & 0 & x_{12} & 0 \cr 0 & 0 & x_{11} & 0 & 0 & x_{12} \cr
 x_{21} & 0 & 0 & x_{22} & 0 & 0 \cr 0 & x_{21} & 0 & 0 & x_{22} & 0 \cr 0 & 0 & x_{21} & 0 & 0 & x_{22} \end{pmatrix}+
\begin{pmatrix} y_{11} & y_{12} & y_{13} & 0 & 0 & 0 \cr y_{21} & y_{22} & y_{23} & 0 & 0 & 0 \cr y_{31} & y_{32} & y_{33} & 0 & 0 & 0 \cr
 0 & 0 & 0 & y_{11} & y_{12} & y_{13} \cr 0 & 0 & 0 & y_{21} & y_{22} & y_{23} \cr 0 & 0 & 0 & y_{31} & y_{32} & y_{33} \end{pmatrix}.
\]  
Now let
\[
v:=\sum_{i,j}v_{ij}(e_i\otimes f_j) \in V_2\otimes V_3.
\]
We can identify this as $v=(v_{11},v_{12},v_{13},v_{21},v_{22},v_{23})$. 
Then $V$ is prehomogeneous if and only if $\{\rho(x+y)(v)\mid x+y\in \Ls\}=V$. This condition can be reformulated
by moving the $v_{ij}$ to the left side and the $x_{ij},y_{ij}$ to the right side as follows:
$V$ is prehomogeneous if and only if
\[
\begin{pmatrix}
v_{11} & v_{21} & 0 & v_{11} & v_{12} & v_{13} & 0 & 0 & 0 & 0 & 0 \cr
v_{12} & v_{22} & 0 & 0 & 0 & 0 & v_{11} & v_{12} & v_{13}  & 0 & 0 \cr
v_{13} & v_{23} & 0 & -v_{13} & 0 & 0 & 0 & -v_{13} & 0 & v_{11} & v_{12} \cr
-v_{21} & 0 & v_{11} & v_{21} & v_{22} & v_{23} & 0 & 0 & 0 & 0 & 0 \cr
-v_{12} & 0 & v_{12} & 0 & 0 & 0 & v_{21} & v_{22} & v_{23} & 0 & 0 \cr
-v_{23} & 0 & v_{13} & -v_{23} & 0 & 0 & 0 & -v_{23} & 0 & v_{21} & v_{22} 
\end{pmatrix}    
\cdot
\begin{pmatrix} x_{11} \cr x_{12} \cr x_{21} \cr y_{11} \cr y_{12} \cr y_{13} \cr y_{21} \cr y_{22} \cr y_{23} \cr
y_{31} \cr y_{32} \end{pmatrix}
\]  
is all of $\C^6$ when $x_{ij},y_{ij}$ run through $\C$. Here we have replaced $x_{22}$ and $y_{33}$ by $x_{22}=-x_{11}$ and
$y_{33}=-y_{11}-y_{22}$. Denoting the $6\times 11$-matrix by $M_v$ we have that $V$ is prehomogeneous if and only if the rank
of $M_v$ equals $6$ for some $v$. It is easy to see that we can find such vectors $v$, for example $v=(0,0,1,0,1,0)$.

\section{The solvable radical of disemisimple Lie algebras}

We have seen that the solvable radical of a disemisimple Lie algebra $\Lg$ in low dimension is in fact {\em abelian}. We ask
ourselves whether or not this is true in general for all disemisimple Lie algebras. In this section we will prove this for the
case that $\Lg$ has a {\em simple} Levi subalgebra. Here we can use Vinberg's classification. However, this proof cannot be
generalized to the semisimple case this way, because prehomogeneous modules for semisimple Lie algebras do not admit a classification
of such an explicit nature as Vinberg's classification. 

\begin{defi}\label{3.1}
Let $\Ls$ be a semisimple Lie algebra and $V$ be an $\Ls$-module. \\[0.1cm]
We say that $V$ is {\em of type $1$}, if $V=A\oplus B$ is the sum of two irreducible non-trivial $\Ls$-submodules $A$ and $B$
of $V$ such that $B$ can be embedded as an $\Ls$-submodule of the exterior product $\Lambda^2(A)=A\wedge A$. \\[0.1cm]
We say that $V$ is {\em of type $2$}, if $V=A\oplus B\oplus C$ is the sum of three irreducible non-trivial $\Ls$-submodules
$A,B,C$ of $V$ such that $C$ can be embedded as an $\Ls$-submodule of the tensor product $A\otimes B$. 
\end{defi}  

We have the following result, which reduces our question on the solvable radical of a di\-semisimple Lie algebra to
the existence of prehomogeneous modules of type $1$ or $2$.

\begin{thm}\label{3.2}
Let $\Ls$ be a semisimple Lie algebra. Then there exists a disemisimple Lie algebra $\Lg$ with a non-abelian nilradical
$\nil(\Lg)$ and Levi subalgebra $\Ls$ if and only if there is a prehomogeneous $\Ls$-module of type $1$ or $2$.   
\end{thm}  

\begin{proof}
Assume first that $\Lg$ is a disemisimple Lie algebra of minimal dimension having non-abelian nilradical and Levi subalgebra $\Ls$.
By Theorem $\ref{2.7}$ we have $\Lg\cong \Ls \ltimes \nil(\Lg)$ such that $\nil(\Lg)$ is a prehomogeneous $\Ls$-module.
Let us write $\Ln$ for $\nil(\Lg)$ in this proof. We will show that $\Ln$ is indeed of type $1$ or $2$. \\
Let $\La \subseteq [\Ln,\Ln]$ be a nonzero ideal of $\Lg$. Then the quotient Lie algebra $\Lg/\La$ has smaller dimension than
$\Lg$, but still is disemisimple with Levi decomposition $\Lg/\La\cong \Ls\ltimes (\Ln/\La )$. But since $\Lg$ was of minimal
dimension having a non-abelian ideal, we must have
$[\Ln/\La,\Ln/\La]=0$, which implies
\[
[\Ln,\Ln]\subseteq \La \subseteq [\Ln,\Ln]
\]
and hence equality. In particular, $\Lg$ has no proper nonzero ideal contained in $[\Ln,\Ln]$. Therefore
$\Ln$ has nilpotency class exactly two and $[\Ln,\Ln]$ is an irreducible $\Ls$-module.
By Weyl's theorem there exists an  $\Ls$-submodule $W$ of $\Ln$ such that $\Ln=W\oplus [\Ln,\Ln]$. Here
$W$ is nonzero, since otherwise $\Ln=0$, which is a contradiction.
Then $W$ is a minimal generating subspace of $\Ln$ as a Lie algebra with $[W,W]=[\Ln,\Ln]$, and it is
a direct sum of irreducible $\Ls$-modules. \\
We claim that the $\Ls$-module $W$ decomposes into {\em exactly one or two irreducible} $\Ls$-modules.
To see this, suppose that we have $W=W_1\oplus W_2\oplus W_3$ for three nonzero $\Ls$-modules. Then the Lie algebras
$\Ls \ltimes (W_i\oplus W_j\oplus [\Ln,\Ln])$ for $(i,j)=(1,2),(1,3),(2,3)$ are disemisimple
having Levi subalgebra $\Ls$. Since $\Lg$ is of minimal dimension, $[W_i\oplus W_j, W_i\oplus W_j]=0$ for these $(i,j)$.
But then also $[\Ln,\Ln]=[W,W]=0$. This is impossible, because $\Ln$ is not abelian. \\
Since by assumption $\Ln$ is a prehomogeneous $\Ls$-module, it cannot contain the trivial $1$-dimensional $\Ls$ module.
In particular, $W$ is a non-trivial $\Ls$-module. We can now distinguish the following cases. \\[0.2cm]
{\em Case 1:} $W$ is a non-trivial irreducible $\Ls$-module. Then $\Ln$ is prehomogeneous of type $1$ with
$A:=W$ and $B:=[\Ln,\Ln]$. Indeed, $\Ln=A\oplus B$ is the sum of two irreducible non-trivial $\Ls$-modules of $\Ln$ such
that $B$ can be embedded as a submodule of $A\wedge A$. To see the last claim,
note that the map $A\wedge A \ra [A,A]$ given by $a\wedge b \mapsto [a,b]$ is a surjective $\Ls$-module homomorphism, which
is well-defined by the universal property of the exterior product. An $\Ls$-module is a quotient of $A\wedge A$ if and only if it
embeds in $A\wedge A$ since $\Ls$ is semisimple. Hence $B=[\Ln,\Ln]=[A,A]$ can be embedded into the
$\Ls$-module $A\wedge A$ and we are done. \\[0.2cm]
{\em Case 2:} $W=A\oplus B$ is the sum of two non-trivial irreducible $\Ls$-modules. The subalgebra $\Ls\ltimes (A+[\Ln,\Ln])$
of $\Lg$ has Levi subalgebra $\Ls$ and nilradical $A+[\Ln,\Ln]$, which is a prehomogeneous $\Ls$-module. Hence this subalgebra
is disemisimple by Theorem $\ref{2.7}$. Again by the minimality of $\Lg$, the nilradical of $\Ls\ltimes (A+[\Ln,\Ln])$
has to be abelian. This implies $[A,A]=0$. In the same way we obtain $[B,B]=0$. Now let $C:=[\Ln,\Ln]$. 
Then we have
\[
C=[\Ln,\Ln]=[W,W]=[A\oplus B,A\oplus B]=[A,B].
\]
We have $\Ln=A\oplus B\oplus C$, where all three submodules are nontrivial and irreducible.
The map $A\otimes B \ra [A,B]$ is again a well-defined surjective $\Ls$-homomorphism. Hence
$C=[A,B]$ can be embedded into the $\Ls$-module $A\otimes B$. It follows that $\Ln$ is a prehomogeneous
module of type $2$. \\[0.2cm]
For the converse statement we also make a case distinction. \\[0.2cm]
{\em Case 1}: $V$ is a prehomogeneous $\Ls$-module of type $1$, i.e., we have $V=A\oplus B$ with non-trivial
irreducible $\Ls$-submodules
$A$ and $B$ of $V$. By assumption $B$ can be embedded as a submodule of $A\wedge A$. Let $\Lf:=A\oplus [A,A]$ be the free-nilpotent
Lie algebra of nilpotency class two on the generating space $A$. Then $\Lf$ is also naturally an $\Ls$-module.
The $\Ls$-submodule $[A,A]=[\Lf,\Lf]$ of $\Lf$ is
isomorphic to the $\Ls$-module $A\wedge A$. Therefore we may assume that $B$ is a submodule of $[A,A]$. By Weyl's theorem there
exists an $\Ls$-module $U$ such that $[A,A]=B\oplus U$. Then $U$ can be considered as an ideal of the Lie algebra $\Ls \ltimes \Lf$.
Now consider the quotient Lie algebra
\[
\Lg=(\Ls \ltimes \Lf)/U \cong \Ls \ltimes (\Lf/U).
\]
It has Levi subalgebra $\Ls$ and its solvable radical is given by
\[
\Lf/U\cong (A\oplus [A,A])/U \cong (A+U)/U \oplus [A,A]/U \cong A/(A\cap U)\oplus B \cong A\oplus B=V.
\]
We have $[\Lf/U,\Lf/U]\cong [\Lf,\Lf]/U\cong [A,A]/U \cong B$, so that the radical $\Lf/U$ is non-abelian and $2$-step nilpotent,
so nilpotent of class exactly two, and is prehomogeneous as an $\Ls$-module by assumption. Hence by Theorem $\ref{2.7}$ the Lie algebra
$\Lg\cong \Ls \ltimes V$ is disemisimple. \\[0.2cm]
{\em Case 2:}  $V$ is a prehomogeneous $\Ls$-module of type $2$, i.e., we have $V=A\oplus B \oplus C$ with non-trivial irreducible
$\Ls$-submodules $A,B,C$ of $V$. Let $\Lf:=(A\oplus B)\oplus [A\oplus B,A\oplus B]$ be the free-nilpotent Lie algebra of
nilpotency class two on the generating space $A\oplus B$. Then $\Lf$ is also naturally an $\Ls$-module and we have
$[A,A]\cong A\wedge A$, $[B,B]\cong B\wedge B$ and $[A,B]\cong A\otimes B$ as $\Ls$-submodules of $\Lf$. Hence we have
\[
[A\oplus B,A\oplus B]\cong [A,A]\oplus [A,B]\oplus [B,B]\cong (A\wedge A)\oplus (A\otimes B)\oplus (B\wedge B).
\]
By assumption $C$ can be embedded as a submodule of $A\otimes B$. We may therefore assume that $C$ is an $\Ls$-submodule of
$[A\oplus B,A\oplus B]$. By Weyl's theorem there exists an $\Ls$-submodule $U$ of $[A\oplus B,A\oplus B]$ such that
$[A\oplus B,A\oplus B]=C\oplus U$. This $U$ is an ideal of the Lie algebra $\Ls\ltimes \Lf$ and we consider the quotient
Lie algebra
\[
\Lg=(\Ls \ltimes \Lf)/U \cong \Ls \ltimes (\Lf/U).
\]  
We have  $[\Lf/U,\Lf/U]\cong [\Lf,\Lf]/U\cong C$, which is nonzero. Hence $\Lg$ has Levi subalgebra $\Ls$ and a solvable
radical of nilpotency class two. As an $\Ls$-module, $\Lf/U\cong A\oplus B\oplus C=V$, which is prehomogeneous by assumption.
Hence by Theorem $\ref{2.7}$ the Lie algebra $\Lg\cong \Ls \ltimes V$ is disemisimple. 
\end{proof}  

For simple Lie algebras we are able to decide whether or not there is a prehomogeneous module of type $1$ or $2$. Let us first
consider an easy example.

\begin{ex}
The Lie algebra $\mathfrak{sl}_2(\C)$ has no prehomogeneous module of type $1$ or $2$. 
\end{ex}  

Suppose that $\Ls=\mathfrak{sl}_2(\C)$ has a prehomogeneous $\Ls$-module $V$ of type $1$ or $2$. Then $V$ is 
the direct sum of two or three nonzero non-trivial irreducible submodules. By Corollary $2.6$ we have 
$2\le \dim(V) \le \dim(\Ls)-1=2$, which is impossible, because the two irreducible submodules have already 
dimension $n\ge 2$ each. Hence the claim follows. Of course, we also know by Vinberg's result given in 
Theorem $\ref{2.10}$ that a simple Lie algebra of type $A_1$ has no reducible prehomogeneous module.

\begin{lem}\label{3.4}
Let $\Ls$ be a simple Lie algebra of type $A_n$ with $n\ge 2$. Then we have the following isomorphisms for the tensor product
of highest weight modules:
\begin{align*}
L(\om_1)\otimes L(\om_1) & \cong L(\om_2) \oplus L(2\om_1),\\
L(\om_n) \otimes L(\om_n) & \cong L(\om_{n-1})\oplus L(2\om_n).
\end{align*}
\end{lem}  

\begin{proof}
For the natural $\Ls$-module $V=L(\om_1)$ we have $V\otimes V\cong \Lambda^2(V)\oplus {\rm Sym}^2(V)$. For the summands we have
$\Lambda^2(V)\cong L(\om_2)$ and $ {\rm Sym}^2(V)\cong L(2\om_1)$, see Exercise $15.32$ in \cite{FUH}. This shows the first
isomorphism. The second one follows from duality, because $L(\om_1)^*\cong L(\om_n)$ and $L(\om_2)^*\cong L(\om_{n-1})$. 
\end{proof}  

\begin{lem}\label{3.5}
Let $\Ls$ be a simple Lie algebra of type $A_{n}$ with $n\ge 4$ even. Then we have the following isomorphisms for the exterior
product of highest weight modules:
\begin{align*}
L(\om_2)\wedge L(\om_2) & \cong L(\om_1+\om_3),\\
L(\om_{n-1}) \wedge L(\om_{n-1}) & \cong L(\om_n+\om_{n-2}).
\end{align*}
\end{lem}  

\begin{proof}
For the natural $\Ls$-module $V=L(\om_1)$ we have $L(\om_2)\cong \Lambda^2(V)$ and
$\Lambda^2(V)\wedge \Lambda^2(V)\cong L(\om_1+\om_3)$, see again Exercise $15.32$ in \cite{FUH}. The second isomorphism follows
by duality.
\end{proof}  

\begin{lem}\label{3.6}
Let $\Ls$ be a simple Lie algebra. Then $\Ls$ has no prehomogeneous module of type $1$ or $2$.
\end{lem}

\begin{proof}
Suppose that $\Ls$ has a prehomogeneous module $V$ of type $1$ or $2$. Since $V$ is reducible, the classification
given in Theorem $\ref{2.10}$ shows that $\Ls$ is of type $A_{\ell}$ for some $\ell \ge 2$. \\[0.2cm]
{\em Case 1:} $V$ is of type $2$, so $V=A\oplus B\oplus C$ is the sum of three irreducible non-trivial $\Ls$-submodules
of $V$ such that $C$ can be embedded as an $\Ls$-submodule of the tensor product $A\otimes B$.
By Theorem $\ref{2.10}$, we obtain only two possibilities for the triple $(A,B,C)$, namely
\begin{align*}
(A,B,C) & = (L(\om_1),L(\om_1),L(\om_1)), \text{ or } \\
(A,B,C) & = (L(\om_{\ell}),L(\om_{\ell}),L(\om_{\ell})).
\end{align*}  
In each case $C$ is not an $\Ls$-submodule of the tensor product $A\otimes B$ by Lemma $\ref{3.4}$.
This is a contradiction to our assumption. \\[0.2cm]
{\em Case 2:}  $V$ is of type $1$, so $V=A\oplus B$ is the sum of two irreducible non-trivial $\Ls$-submodules of $V$ such
that $B$ is an $\Ls$-submodule of the exterior product $\Lambda^2(A)=A\wedge A$. Then the classification given in
Theorem $\ref{2.10}$ shows that either  $\Ls$ is of type $A_{\ell}$ for some $\ell \ge 2$ and
\begin{align*}
(A,B) & = (L(\om_1),L(\om_1)), \text{ or } \\
(A,B) & = (L(\om_{\ell}),L(\om_{\ell}))),
\end{align*}
or $\Ls$ is of type $A_n$ for some even $n=2\ell \ge 4$ and
\begin{align*}
(A,B) & = (L(\om_1),L(\om_{n-1})), (L(\om_2),L(\om_2)), (L(\om_2),L(\om_n)),  \text{ or } \\ 
(A,B) & = (L(\om_n),L(\om_2)), (L(\om_{n-1}),L(\om_2)), (L(\om_{n-1}),L(\om_{n-1})).
\end{align*}          
By assumption we have the inclusions $B\le A\wedge A\le A\otimes A$ of $\Ls$-modules. However, this is in all cases
above not true by Lemma $\ref{3.4}$ and $\ref{3.5}$. So we obtain a contradiction and we are done.
\end{proof}

These lemmas now yield our main result of this section.

\begin{thm}\label{3.7}
Let $\Lg$ be a disemisimple Lie algebra having a simple Levi subalgebra $\Ls$. Then $\rad(\Lg)$ is abelian.
\end{thm}

\begin{proof}
Assume that  $\rad(\Lg)$ is non-abelian. Then $\Lg$ is disemisimple by Theorem $\ref{3.2}$ if and only if $\rad(\Lg)$
is a prehomogeneous $\Ls$-module of type $1$ or $2$. Since $\Ls$ is simple, it has no such prehomogeneous $\Ls$-module
according to Lemma $\ref{3.6}$. This is a contradiction. Hence  $\rad(\Lg)$ is abelian.  
\end{proof}  

\section{Simple factors of type A}

In this section we want to generalize Theorem $\ref{3.7}$ to disemisimple Lie algebras having no simple quotient
of type $A$.

\begin{defi}
A disemisimple Lie algebra $\Lg$ is called {\em $A$-free}, if it has no simple quotient of type $A$.
\end{defi}

Note that a disemisimple Lie algebra $\Lg$ is $A$-free if and only if its Levi subalgebras have no simple factor of type $A$.
We need the following lemma.

\begin{lem}\label{4.2}
Let $\Ls$ be a semisimple Lie algebra and $V$ be a nonzero irreducible prehomogeneous $\Ls$-module acting by
a Lie algebra homomorphism $\rho\colon \Ls\ra \End(V)$. If $\Ls$ has no simple factor of type $A$ then
$\Ls/\ker(\rho)$ is simple of type $C_{\ell}$ or $D_5$.
\end{lem}  

\begin{proof}
This follows from the results of Sato and Kimura in \cite{SAK}. They are formulated for semisimple algebraic groups, but can
be reformulated in terms of semisimple Lie algebras. See also the remark after Proposition $12$ on page $40$ in \cite{SAK}.
So for a semisimple Lie algebra $\Ls$ let us write $\Ls=\CA (\Ls)\oplus \CN (\Ls)$, where $\CA (\Ls)$ is the direct sum of
simple factors of type $A_{\ell}$, and $\CN (\Ls)$ the direct sum of simple factors, which are not of type $A_{\ell}$. For any nonzero
irreducible $\Ls$-module $V$ there exist an irreducible $\CA (\Ls)$-module $V_{\CA(\Ls)}$ and an irreducible
$\CN$-module $V_{\CN(\Ls)}$ such that $V\cong V_{\CA (\Ls)}\otimes V_{\CN (\Ls)}$.
Let $(\Lg,\rho,V)$ and $(\Lh,\sigma,W)$ be castling transforms of each other, see Definition $10$ on page $39$ in \cite{SAK},
where $\Lg$ and $\Lh$  are semisimple Lie algebras. Then $V_{\CN (\Lg)}$ is isomorphic to $W_{\CN (\Lh)}$ or to $W_{\CN (\Lh)}^*$,
so that $\dim (V_{\CN (\Lg)})=\dim (W_{\CN (\Lh)})$. \\[0.2cm]
Suppose that  $(\Lh,\sigma,W)$ is a prehomogeneous triple with $\CA(\Lh)=0$. Then it is castling equivalent to a reduced
triple  $(\Lg,\rho,W)$ of the classification list of Proposition $\ref{2.13}$. Since  $(\Lg,\rho,W)$ is reduced
we have the inequality
\begin{align*}
\dim (V_{\CA(\Lg)}) \cdot \dim (W_{\CN(\Lh)}) & = \dim (V_{\CA(\Lg)}) \cdot \dim (V_{\CN(\Lg)}) \\
                                            & = \dim (V)\le \dim (W) \\
                                            & =\dim (W_{\CN(\Lh)}),
\end{align*}  
and therefore $\dim (V_{\CA(\Lg)})=1$ and $\dim (V)=\dim (W)$. Hence the triple $(\Lh,\sigma,W)$ is itself reduced by definition.
Thus it is strongly equivalent to $(\Lg,\rho,W)$, because a reduced triple in any castling class is unique up to strong
equivalence by Proposition $12$ on page $39$ in \cite{SAK}. This implies that $\rho(\Lg)\cong \sigma(\Lh)$. Here
$\rho$ is faithful and $\sigma(\Lh)$ has no quotients of type $A$, because $\sigma(\Lh)$ is a quotient of $\Lh$, which has no
quotients of type  $A$ by assumption. So also $\Lg$ has no simple factor of type  $A$. The list of
Proposition $\ref{2.13}$ now implies that $\Lg$ itself is simple, and is in fact of type $C_{\ell}$ or $D_5$. Hence
\[
H/\ker(\sigma) \cong \sigma (\Lh)\cong \rho(\Lg)\cong \Lg
\]
is simple of type  $C_{\ell}$ or $D_5$.
\end{proof}  

Now we can prove our main result of this section.

\begin{thm}\label{4.3}
Let $\Lg$ be a disemisimple, $A$-free Lie algebra. Then its solvable radical is abelian. 
\end{thm}  

\begin{proof}
Let $\Ls=\Ls_1\oplus \cdots \oplus \Ls_k$ be a Levi subalgebra of $\Lg$ with simple factors $\Ls_i$ and assume that
$\rad(\Lg)$ is non-abelian. By Theorem $\ref{3.2}$ there exists a prehomogeneous $\Ls$-module $V$ of type $1$ or $2$. \\[0.2cm] 
{\em Case 1:} $V=A\oplus B$ is of type $1$. Then by Lemma $\ref{4.2}$ there exist unique $1\le i,j\le k$ such that $\Ls_i$ acts
only non-trivially on $A$ and $\Ls_j$ acts only non-trivially on $B$. By definition the ideal $\bigoplus_{\ell \neq i} \Ls_{\ell}$ of $\Ls$
acts trivially on $A$, and therefore trivially of $A\wedge A$. So $i=j$ and we see in particular that $V$ is a prehomogeneous
$\Ls_i$-module of type $1$. Since $\Ls_i$ is simple, we obtain a contradiction to Lemma $\ref{3.6}$. \\[0.2cm]
{\em Case 2:} $V=A\oplus B\oplus C$ is of type $2$. Then by Lemma $\ref{4.2}$ there exist unique $1\le p,q,r \le k$
such that $\Ls_p$ acts only non-trivially on $A$,  $\Ls_q$ acts only non-trivially on $B$, and  $\Ls_r$ acts only non-trivially on $C$.
Then the ideal $\bigoplus_{\ell \neq p,q} \Ls_{\ell}$ of $\Ls$ acts trivially on $A$ and $B$, and hence trivially on $A\otimes B$ and
its submodule $C$. So $V$ is a prehomogeneous  $(\Ls_p+ \Ls_q)$-module of type $2$. We claim that $p=q$ so that $\Ls_p+ \Ls_q=\Ls_p$.
Suppose we have $p\neq q$. Then $\Ls_p$ acts non-trivially and irreducibly on $A$, and trivially on $B$. Similarly,
$\Ls_q$ acts non-trivially and irreducibly on $B$, and trivially on $A$. So the ideals  $\Ls_p$ and $\Ls_q$ act non-trivially
and irreducibly on $A\otimes B$. Since $C$ is a non-trivial  irreducible submodule, we obtain $C=A\otimes B$. Hence
both $\Ls_p$ and $\Ls_q$ act non-trivially on $C$, which means $r=p$ and $r=q$, a contradiction. For $p=q$ however, 
$V$ is a  prehomogeneous $\Ls_p$-module of type $2$, where $\Ls_p$ is simple. This is a again a contradiction to Lemma $\ref{3.6}$.
\end{proof}  

Let $\Ls$  be a semisimple Lie algebra and $V$ be an $\Ls$-module. Denote by $\Lg_{\Ls,V}=\Ls \ltimes V$ be the semidirect product
given by
\[
[(s,v),(t,w)]=([s,t],s\cdot w-t\cdot v)
\]  
for $s,t\in \Ls$ and $v,w\in V$. The solvable radical of $\Lg_{\Ls,V}$ is  given by $V$ as a vector space and is abelian as a Lie algebra.
We obtain the following result.

\begin{prop}
Let $\Lg$ be a disimisimple, $A$-free Lie algebra. Then we have
\[
\Lg\cong \Lg_{\Ls_1,V_1}\oplus \cdots \oplus \Lg_{\Ls_k,V_k}
\]
for some $k\ge 1$, where $\Ls_i$ is simple and $V_i$ is a prehomogeneous $\Ls_i$-module.
\end{prop}  

\begin{proof}
Let $\Ls=\Ls_1\oplus \cdots \oplus \Ls_k$ be the decomposition into simple ideals of a Levi subalgebra $\Ls$, and
$\rad(\Lg)=V_1\oplus  \cdots \oplus V_{\ell}$ be the decomposition of $\rad(\Lg)$ into irreducible $\Ls$-modules $V_1,\ldots ,V_{\ell}$.
By Lemma $\ref{4.2}$ there exists for each $V_i$ a unique $\sigma(i)$ such that $\Ls_{\sigma(i)}$ acts non-trivially on $\Ls_i$.
Since $\Lg$ is $A$-free, the classification of Vinberg, Theorem $\ref{2.9}, \ref{2.10}$, shows that for each $\Ls_j$ there is at most
one $V_{\pi(j)}$ such that $\Ls_j$ acts non-trivially on $V_{\pi(j)}$. So after reordering the simple ideals $\Ls_1,\ldots ,\Ls_k$ we may
assume that $\Ls_1,\ldots ,\Ls_{\ell}$ act non-trivially on $V_1,\ldots , V_{\ell}$ respectively, and that the other simple ideals
$\Ls_j$ for $\ell+1\le j\le k$ act trivially on $\rad(\Lg)$. By setting $V_j=0$ for such $j$, we obtain the claimed isomorphism of Lie algebras.
\end{proof}

\section*{Acknowledgments}
The authors are supported by the Austrian Science Foun\-da\-tion FWF, grant I 3248 and grant P 33811.
The second author was partly supported by the FWF grant P 30842.

\end{document}